\documentclass[twoside,11pt]{amsart}

\usepackage{amsmath,latexsym,amssymb,mathptm, times,verbatim, enumerate,times}

\input amssym.def
\input amssym
\input xypic
\input xy

\xyoption{all}
\def\Pf{\operatorname{Pf}}
\def\topdeg{\operatorname{topdeg}}
\def\reg{\operatorname{reg}}

\def\grade{\operatorname{grade}}

\def\Sym{\operatorname{Sym}}
\define\Hom{\operatorname{Hom}}

\def\Supp{\operatorname{Supp}}

\define\coker{\operatorname{coker}}

\def\HH{\operatorname{H}}
\def\depth{\operatorname{depth}}
\def\indeg{\operatorname{indeg}}
\def\Ext{\operatorname{Ext}}
\def\socle{\operatorname{socle}}
\def\maxgendeg{b_0}
\def\mgd{b_0}

\def\p{\mathfrak p}

\def\m{\mathfrak m}

\def\htt{\operatorname{ht}}
\def\lto{\longrightarrow}

\newtheorem{theorem}{Theorem}[section]
\newtheorem{lemma}[theorem]{Lemma}
\newtheorem{corollary}[theorem]{Corollary}
\newtheorem{proposition}[theorem]{Proposition}

\newtheorem{observation}[theorem]{Observation}

\newtheorem*{Corollary5.7}{Corollary \ref{SJD-cor3}}
\newtheorem*{Theorem10.1}{Theorem \ref{Main10}}

\theoremstyle{definition}
\newtheorem{def&dis}[theorem]{Definition and Discussion}

\newtheorem{remark}[theorem]{Remark}

\newtheorem{chunk}[theorem]{}
\newtheorem{example}[theorem]{Example}
\newtheorem{setup}[theorem]{Setup}

\newtheorem*{Reminder}{Reminder}
\numberwithin{equation}{theorem}
\newenvironment{outline of proof}{\paragraph{\em Outline of proof.}}{\hfill$\qed$}

\begin{document}

\baselineskip=16pt

 \title[Degree bounds for  local cohomology]{\bf Degree bounds for  local cohomology}
\date\today

\author[Andrew R. Kustin, Claudia Polini, and Bernd Ulrich]
{Andrew R. Kustin, Claudia Polini, and Bernd Ulrich}

\thanks{AMS 2010 {\em Mathematics Subject Classification}.
Primary 13D45, 13C40, 14A10, 13D02, 13A30.
}

\thanks{The first author was partially supported by the Simons Foundation.
The second and third authors were partially supported by the NSF}


\address{Department of Mathematics, University of South Carolina,
Columbia, SC 29208} \email{kustin@math.sc.edu}

\address{Department of Mathematics, 
University of Notre Dame
Notre Dame, IN 46556} \email{cpolini@nd.edu}

\address{Department of Mathematics,
Purdue University,
West Lafayette, IN 47907}\email{bulrich@purdue.edu}

 \begin{abstract}
 
 It has long been known how to read information about the socle degrees of the local cohomology $\HH_\mathfrak m^0(M)$ of a graded $R$-module from the  twists in position $d=\dim R$, in a resolution of  $M$ by free $R$-modules.  It has also long been known how to use local cohomology to read valuable information from complexes which approximate resolutions in the sense that they have positive homology of small Krull dimension. 
The 
present paper reads information about the maximal generator degree (rather than the socle degree) of $\HH_\mathfrak m^0(M)$ from the 
 twists in position $d-1$ (rather than position $d$) in an approximate resolution of $M$. 

We apply the local cohomology results to draw conclusions about  the maximum generator degree of the second symbolic power of the prime ideal defining a monomial curve and the second symbolic power of the ideal defining a finite set of points in projective space. There is also an application to hyperplane sections of subschemes of projective space and  to partial Castelnuovo-Mumford  regularity.
 Perhaps, the most important application 
is to the study of blow-up algebras and their defining equations. 
The techniques of the present paper are the main tool used in \cite{KPU-BA} to 
bound the degrees of these equations 
and thus to identify them  in some cases.

\end{abstract}
 
\maketitle


\section{Introduction}\label{Intro}

\bigskip
For the time being, let $R$ be a non-negatively graded polynomial ring in $d$ variables  over a field and $M$ be  a finitely generated graded  $R$-module of depth zero. It is well known how to \begin{equation}\label{*}\begin{array}{l}\text{read the socle degrees of $M$ from the twists at the end of}\\\text{a minimal homogeneous finite free resolution of $M$.}\end{array}\end{equation}
 Of course, the socle of $M$ is the socle of the local cohomology module $\HH^0_\mathfrak m(M)$, where $\mathfrak m$ is the maximal homogeneous ideal of $R$.  In this paper we find bounds
on  the degrees of interesting elements of $\HH^i_\mathfrak m(M)$ in terms of information about the ring $R$ and information that can be read from a homogeneous complex of finitely generated  $R$-modules  $C_{\bullet}: \cdots \to C_2\to C_1\to C_0\to 0$ with $\HH_0(C_{\bullet})=M$.
The ring $R$ need not be a polynomial ring, the complex $C_\bullet$ need not be finite, need not be acyclic, and need not consist of free modules, and the parameter $i$ need not be zero. Instead, we impose hypotheses on the Krull dimension of $\HH_j(C_\bullet)$ and the depth of $C_j$ in order make various local cohomology modules $\HH_\mathfrak m^\ell(\HH_j(C_\bullet))$  and $\HH_\mathfrak m^\ell(C_j)$ vanish.

The crucial technical result is Proposition~\ref{SJD-P}. 
In Theorem~\ref{SJD-cor1}, our main theorem, we bound the maximal generator degree of $\HH^i_\mathfrak m(M)$ in terms of the maximal generator degree of $C_j$ for appropriately related $i$ and $j$. 
In particular, in Corollary~\ref{SJD-cor2}, we bound  the maximal generator degree of $\HH^0_\mathfrak m(M)$ in terms of the maximal generator degree of $C_{d-1}$. The hypotheses of  Corollary~\ref{SJD-cor2} hold if $C_\bullet$ is a free resolution of $M$; consequently, this result is completely analogous to (\ref{*}) where $\max\{r\mid [\HH^0_\mathfrak m(M)]_r\neq 0\}$ is read from the generator degrees of $C_d$. Corollary~\ref{SJD-cor3} is an intriguing generalization of the well-known fact that a maximal Cohen-Macaulay module over a polynomial ring is free. 

Corollary~\ref{SJD-cor2} is precisely the result that we use in \cite{KPU-BA}  to identify the torsion submodule of the symmetric powers $\Sym_\ell(I)$ where $I$ is a grade three Gorenstein ideal in an even-dimensional polynomial ring. 
A more elementary result (Proposition~\ref{localCoh}) may be used to identify the torsion submodule of $\Sym_\ell(I)$ when $I$ is a grade three Gorenstein ideal in an odd-dimensional polynomial ring. 
We view Proposition~\ref{localCoh} as a model for the main results in the present paper.

In Section~\ref{AEC}, using a spectral sequence argument, we relate the cohomology of $\Hom(C_{\bullet},N)$ to $\Ext^\bullet(M,N)$, where $C_\bullet$ is a complex with $\HH_0(C_\bullet)=M$ and $M$ and $N$ are arbitrary modules. In spite of the a priori lack of hypotheses we obtain a significant, multi-faceted, result  which is the basis for Section~\ref{TMAS}; and hence the rest of the paper.

In Section~\ref{dim-one},
we apply the local cohomology techniques of Section~\ref{TMAS} to draw conclusions about geometric situations. 
Corollary~\ref{monomialCurve} shows that if 
$\mathfrak p$  
is the prime ideal which 
defines the monomial curve associated to a
numerical semigroup $H$,
then the maximal generator degree of the second symbolic power of $\mathfrak p$ satisfies
$$\maxgendeg (\mathfrak p^{(2)})\le \sup\{\maxgendeg (\mathfrak p)+\text{ the maximal generator of $H$} +\text{ the Frobenius number of $H$}, 2\mgd (\mathfrak p)\}.$$ Corollary~\ref{typed-by-Claudia} shows that if $I$ is the ideal which defines a finite set of points in projective space, then
$$b_0(I^{(2)})\le b_0(I)+p(P/I) +2,$$  where $p(P/I)$ is the postulation number of the homogeneous coordinate ring of the set of points. 
Corollary~\ref{hyperplane} is about hyperplane sections of subschemes of projective space. Let $V$ be the subscheme of $\mathbb P^{d-1}_k$ defined by the homogeneous ideal $I$ in $R=k[x_1,\dots,x_d]$ and $H$ be a linear subspace of $\mathbb P^{d-1}_k$ defined by general linear forms in $k[x_1,\dots,x_d]$. We produce an upper bound for the maximal generator degree of the saturated ideal defining the subscheme  $V\cap H$ of $H$, in terms of information that can be read from a single shift in the minimal homogeneous resolution of $R/I$. 

\bigskip\noindent{\bf Acknolwedgment.}  We are indebted to the referee of the first version of this paper. Owing to the referee's suggestions, the proof of Lemma~\ref{SPD} was simplified substantially and the assumptions of Theorem~\ref{SJD-cor1} were weakened. 


\section{Conventions,  notation, and preliminary results}\label{ConNotPrel}

\begin{chunk} By a graded ring or module we mean a $\mathbb Z$-graded ring or module, unless otherwise specified. Notice that every ring and module is graded, namely trivially graded. 
\end{chunk}

\begin{chunk} If $M$ and $N$ are graded modules over  a graded  ring $R$, then $N$ is a homogeneous  {\it subquotient} of $M$ if $N$ is isomorphic to a homogeneous submodule of a graded homomorphic image of $M$; that is, if there exists a graded  $R$-module $P$ with 
$$\xymatrix{
&M\ar@{->>}[d]\\
N\ar@{^{(}->}[r]&P.}$$ Of course, $N$ is also a subquotient of $M$ if $N$ is isomorphic to a graded homomorphic image of a homogeneous submodule of $M$. The property of being a subquotient is transitive in the sense that  if $M_1$ is a subquotient of $M_2$ and $M_2$ is a subquotient of $M_3$, then $M_1$ is a subquotient of $M_3$.
\end{chunk}

\begin{chunk} Let $M$ and $N$ be graded modules over a graded ring $R$. By $^*\Hom_R(M,N)$ and $^*\Ext^i_R(M,N)$ we denoted the graded $\Hom$ and $\Ext$ modules. They coincide with the usual $\Hom$ and $\Ext$ modules if $R$ is Noetherian and $M$ is finitely generated or if $R$, $M$, and $N$ are trivially graded.
\end{chunk}

\begin{chunk} Let $M$ and $N$ be graded modules over a graded ring $R$. In order to simplify our formulas we set $\Ext_R^i(M,N)=0$ and   $^*\Ext_R^i(M,N)=0$ for $i< 0$. If $R$ is a non-negatively graded Noetherian ring, $R_0$ is local, and  $\mathfrak m$ is the maximal homogeneous ideal of $R$, we also set $ \HH_{\mathfrak m}^i(M)=0$ for $i< 0$.
\end{chunk}

\begin{chunk}\label{numerical-functions}We collect  names  for some of the invariants associated to a graded module. Let $R$ be a graded  ring and $M$ be a graded $R$-module. Define
\begin{align}
\topdeg M&=\sup\{j \mid M_j\not=0\},
\notag\\
\indeg M&=\inf\{j \mid M_j\not=0\}, 
\notag\\
b_0(M)&\textstyle=\inf \left\{b \, | \,  R\left(\bigoplus_{j\le b}M_j\right)=M \right\}.\notag\\
\intertext{
If $R$ is a non-negatively graded Noetherian ring, $R_0$ is local,  $\mathfrak m$ is the maximal homogeneous ideal of $R$, and $M$ is finitely generated,  then also define}
a_i(M)&=\topdeg \HH_{\mathfrak m}^i(M) \text{ \ and}
\notag\\
b_i(M)&=\topdeg \operatorname{Tor}^R_i(M,R/\mathfrak m).
\notag\end{align}
Observe that both definitions of the maximal generator degree $b_0(M)$ give the same value.  The expressions ``$\topdeg$'', ``$\indeg$'', and ``$b_i$'' are read ``top degree'', ``initial degree'', and 
``maximal $i$-th shift in a minimal homogeneous free resolution'',
respectively.
If $M$ is the zero module, then $$\topdeg(M)=b_0(M)=-\infty\quad\text{and}\quad \indeg M=\infty.$$ In general one has 
$$a_i(M)< \infty \quad\text{and}\quad b_i(M)< \infty\, .$$
\end{chunk}

\smallskip

We often use the  data of \ref{StandardData}. 
\begin{chunk}\label{StandardData}Let $R$ be a non-negatively graded Noetherian ring with $R_0$ local, let $\mathfrak m$ be the maximal homogeneous ideal of $R$, and write $k=R/\m$ for the residue field.\end{chunk}

\begin{chunk}\label{Theta-sub-t}Take  $R,R_0,\mathfrak m$  as described in {\rm\ref{StandardData}}.  For a minimal homogeneous generating set $y_1,\dots,y_n$ of $\mathfrak m$, with $\deg y_1\ge \deg y_2\ge \dots \ge \deg y_n$, let $\Theta_t=\sum_{j=1}^t\deg y_j\,$ if $t\le n$ and $\Theta_t=-\infty\, $ otherwise.\end{chunk}

\begin{chunk}\label{GCM} Take  $R,R_0,\mathfrak m$  as described in {\rm\ref{StandardData}}. Assume further that 
$R_0$ is a 
factor ring of a local Gorenstein ring $T$.
Let $S=T[x_1,\dots,x_n]$ be a graded polynomial ring which maps homogeneously onto $R$. If $g$ is the codimension of $R$ as an $S$-module, 
then the graded canonical module of $R$ is $$\omega_R=\Ext_S^g(R,S)(-\sum\deg x_i).$$
\end{chunk}

\begin{chunk}\label{ainvariant} Take  $R,R_0,\mathfrak m$  as described in {\rm\ref{StandardData}} with $\dim R=d$. Recall the numerical functions of \ref{numerical-functions}. 
 The 
 $a$-invariant of $R$ is 
$$a(R)=a_d(R).$$
Furthermore, if $\omega_R$ is the graded canonical module of $R$ (see \ref{GCM}), then 
$$a(R)=-\indeg \omega_R
.$$
\end{chunk}

\begin{chunk}The graded ring $R=\bigoplus\limits_{i\ge 0}R_i$ is a {\it standard graded $R_0$-algebra} if $R$ is generated as an $R_0$-algebra by $R_1$ and $R_1$ is finitely generated as an $R_0$-module. 
\end{chunk}

\begin{chunk}
\label{aiM} Let  $R$ be a standard graded polynomial ring over a field $k$, $\mathfrak m$ be the maximal homogeneous ideal of $R$, and   
$M$ be a finitely generated graded $R$-module.  Recall the numerical functions of \ref{numerical-functions}. 
The Castelnuovo-Mumford {\it regularity} of  $M$ is
$$\reg M=\sup\{a_i(M)+i\}=\sup\{b_i(M)-i\}.$$
\end{chunk}

\begin{chunk} Let $q$ be an integer and $R$ be a ring. A complex of finitely generated free $R$-modules 
$$ \, \ldots \, \lto C_1\lto C_0\lto 0$$ is called {\it $q$-linear}
if $C_i=R(-q-i)^{\beta_i}$ for some $\beta_i$, for all $i$ with $0\le i$. 
\end{chunk}

\smallskip


\section{First bounds on local cohomology modules}\label{AEC}
The main result of this paper is Theorem~\ref{SJD-cor1}. 
Lemma~\ref{SPD} is the first step in the proof of Theorem~\ref{SJD-cor1}. 
In Lemma~\ref{SPD} we relate the cohomology of $\Hom(C_{\bullet},N)$ to $\Ext^\bullet(M,N)$, where $C_\bullet$ is a complex with $\HH_0(C_\bullet)=M$ and $N$ is an arbitrary module. 

\begin{setup}\label{SU-SPD}Let $R$ be a graded ring, $M$ and $N$ be graded $R$-modules, and $$C_{\bullet}:\quad \dots \longrightarrow C_1\longrightarrow C_0\to 0$$be a homogeneous complex of graded $R$-modules with $\HH_0(C_{\bullet})=M$. Fix an integer $i$. We consider two hypotheses which can be imposed on the above data:
\begin{itemize}
\item The data satisfies $U_{i}$ if $\ ^{*}\Ext^{i-j}(\HH_j(C_\bullet),N)=0\, $   for all integers $j$ with $1\le j\le i$.


\item The data satisfies $V_{i}\, $ if $\ ^{*}\Ext^{i-j}(C_j,N)=0\, $  for all integers $j$  with $0\le j\le i-1$.

\end{itemize}
\end{setup}

\begin{lemma}\label{SPD} In the setup of \ {\rm\ref{SU-SPD}}, the following statements hold.
\begin{enumerate}[\rm(a)]
\item\label{SPD-iso} If the hypotheses $U_{i-1}$, $U_{i}$, $V_{i-1}$, and $V_i$ all are in effect, then there is a natural homogeneous isomorphism $\ \HH^i(^{*}\Hom(C_{\bullet},N))\simeq \,^{*}\Ext^i(M,N)$.
\item\label{SPD-E-onto-H} If the hypotheses $U_{i}$, $V_{i-1}$, and $V_i$  are in effect, then there is a natural homogeneous surjection $$\xymatrix{ ^*\Ext^i(M,N)\ar@{->>}[r]& \HH^i(^{*}\Hom(C_{\bullet},N))}.$$
\item\label{SPD-H-onto-E} If the hypotheses $U_{i-1}$, $U_{i}$, and $V_i$  are in effect, then there is a natural homogeneous surjection $$\xymatrix{\HH^i(^{*}\Hom(C_{\bullet},N))\ar@{->>}[r]&  ^*\Ext^i(M,N)}.$$
\item \label{SPD-H-in-E}If the hypotheses $U_{i-1}$, $U_{i}$, and $V_{i-1}$   are in effect, then there is a natural homogeneous injection $$\xymatrix{
\HH^i(^*\Hom(C_{\bullet},N))
\ \ar@{^{(}->}[r]& ^*\Ext^i(M,N) }.$$
\item\label{SPD-E-in-H} If the hypotheses  $U_{i-1}$, $V_{i-1}$, and $V_i$   are in effect, then there is a natural homogeneous injection $$\xymatrix{
^*\Ext^i(M,N)
\ \ar@{^{(}->}[r]& \HH^i(^*\Hom(C_{\bullet},N)) }.$$
\item\label{SPD-H-sq-of-E} If the hypotheses $U_{i}$ and $V_{i-1}$  are in effect, then $\HH^i(^*\Hom(C_{\bullet},N))$ is a natural  homogeneous subquotient of $\ ^*\Ext^i(M,N)$.

\item\label{SPD-E-sq-of-H} If the hypotheses $U_{i-1}$ and $V_i$   are in effect, then $^*\Ext^i(M,N)$ is a natural  homogeneous subquotient of $\HH^i(^*\Hom(C_{\bullet},N))$.
\end{enumerate}
\end{lemma}
\begin{proof} For $i\le 0$ the assertions are obvious, hence we may assume that $i$ is a positive integer. 
Let $I^{\bullet}$ be a homogenous resolution of $N$ by graded injective modules. Consider the double complex $D^{pq}:= {^*\Hom}(C_q,I^p)$ and write $T^{\bullet}$ for its total complex. The horizontal and vertical filtration of the double complex yield first quadrant spectral sequences whose $E_2$ terms are
\[^{'}E_2^{p,q}\simeq\, ^*\Ext^p(\HH_q(C_{\bullet}),N) \qquad \mbox{and} \qquad ^{''}E_2^{p,q}\simeq \HH^q(^*\Ext^p(C_{\bullet},N))\, ,
\]
respectively. 

The infinity term $^{'}E_{\infty}^{p,q}$ is a homogeneous subquotient of $\, ^{'}E_2^{p,q}$. It is a homogeneous epimorphic image if $q=0$,
\begin{equation}\label{epih} 
\xymatrix{^{'}E_2^{p,0} \ar@{->>}[r]& ^{'}E_{\infty}^{p,0}\, .}
\end{equation}

On the other hand there is a homogeneous embedding 
\begin{equation}\label{embh} 
\xymatrix{^{'}E_{\infty}^{p,0} \ \ar@{^{(}->}[r]& \HH^p(T^{\bullet})\, ,}
\end{equation}
whose cokernel has a  filtration with factors  $\, ^{'}E_{\infty}^{p-q,q}$ for $1\le q\le p$.

The condition $U_{p-1}$ is equivalent to  $\, ^{'}E_2^{p-1-j,j}=0\, $ for $1\le j \le p-1$, which implies that $\, ^{'}E_{j+1}^{p-1-j,j}=0\, $ for $1\le j \le p-1$ and hence $\, ^{'}E_{j+1}^{p-1-j,j}=0\, $ for $1\le j$. It follows that  on the $(j+1)$-st page of the spectral sequence the natural map $\, ^{'}E_{j+1}^{p-1-j,j} \to \, ^{'}E_{j+1}^{p,0}$ is the zero map for $1\le j$. Hence the epimorphism of (\ref{epih}) is an isomorphism. 

The condition $U_p$ means that  $\, ^{'}E_2^{p-q,q}=0\, $ for $1\le q\le p$, which gives $\, ^{'}E_{\infty}^{p-q,q}=0\, $ for $1\le q\le p$. Hence the inclusion of (\ref{embh}) is an isomorphism. 

Likewise,  $^{''}E_{\infty}^{p,q}$ is a homogeneous subquotient of $\, ^{''}E_2^{p,q}$ and a homogeneous epimorphic image if $p=0$. Hence
\begin{equation}\label{epiv} 
\xymatrix{^{''}E_2^{0,q} \ar@{->>}[r]& ^{''}E_{\infty}^{0,q}\, .}
\end{equation}
Also there is a homogeneous embedding
\begin{equation}\label{embv} 
\xymatrix{^{''}E_{\infty}^{0,q}\  \ar@{^{(}->}[r]& \HH^q(T^{\bullet})\, ,}
\end{equation}
and the cokernel of this embedding has a filtration whose factors are $\, ^{''}E_{\infty}^{p,q-p}$ for $1\le p\le q$. 

Now $V_{q-1}$ implies $\, ^{''}E_2^{j-1,q-j}=0\, $ for $2\le j \le q$, which gives $\, ^{''}E_j^{j-1,q-j}=0\, $ for $2\le j$. Therefore the natural map  $\, ^{''}E_j^{j-1,q-j} \to \, ^{''}E_j^{0,q}$ is the zero map for $2\le j$, and hence the epimorphism of (\ref{epiv}) is an isomorphism. 

Finally, $V_q$ gives $^{''}E_2^{p,q-p}=0$ for $1\le p \le q$. Therefore $^{''}E_{\infty}^{p,q-p}=0,$ which means that the embedding of (\ref{embv}) is an isomorphism. 

The lemma now follows from the homogenous epimorphisms and embeddings  (\ref{epih}), (\ref{embh}), (\ref{embv}), (\ref{epiv}) for $p=q=i$ and the various conditions for when they are isomorphisms. 
\end{proof}

\begin{remark} The above proof shows that in Lemma~\ref{SPD} the condition $V_i$ can be replaced by the weaker assumption that
$\HH^j(\Ext^{i-j}(C_{\bullet},N))=0$ for all $j$ with $0\le j \le i-1$, and likewise for $V_{i-1}$. \end{remark}


Observation~\ref{SJD-doo} shows that the hypotheses  of Lemma  \ref{SPD} are implied by some natural assumptions on a complex. 

\begin{setup}\label{SU-SJD} Let $R$ be a non-negatively graded Cohen-Macaulay ring with $R_0$ a 
factor ring of a local Gorenstein ring. Let $d$ be the dimension of $R$ and assume that $1\le d$. Let $\mathfrak m$ be the maximal homogeneous ideal of $R$, $k=R/\mathfrak m$ its residue field, and $\omega=\omega_R$ its graded canonical module; see \ref{GCM}. Let
$$(C_{\bullet},\partial_\bullet):\quad \ \ldots \, \stackrel{\partial_3}\longrightarrow C_2\stackrel{\partial_2}\longrightarrow
C_1\stackrel{\partial_1}\longrightarrow C_0 \longrightarrow  0$$ be a homogeneous complex of finitely generated graded $R$-modules. Write $M=\HH_0(C_{\bullet})$ and $\HH_{\bullet}=\HH_{\bullet}(C_{\bullet})$. Let $(-)^\vee$ denote the functor $\Hom_R(-,\omega)$.
\end{setup}

\begin{observation}\label{SJD-doo} Adopt the setup of \/ {\rm\ref{SU-SJD}} and use the hypotheses $U_i$, $V_i$ of \/ {\rm \ref{SU-SPD}} with $N=\omega$. 
\begin{enumerate}[\rm(a)]
\item\label{SJD-doo-a} Fix an integer $i$. If $\, \dim \HH_j\le d-i+j\, $ for every $j$ with $1\le j\le i-1$, then the data satisfies condition $U_{\ell}\, $ for all $\ell$ with $ \ell\le i-1$. 
\item\label{SJD-doo-b} Fix  integers $r$ and $s$. If $\min\{d,d-r+j+1\}\le \depth C_j\, $ for every $j$ with $0\le j\le s-1$,     then the data satisfies condition $V_{\ell}\, $ for all $\ell$ with $r\le \ell\le s$. 
\end{enumerate}
\end{observation}
\begin{proof} We may assume that the local ring $R_0$ is complete. In the setting of (\ref{SJD-doo-a}), we have $\HH_{\mathfrak m}^k(H_j)=0$ for every $k$ with $d-i+j+1\le k$; that is, $d-(i-j-1)\le k$. By graded duality, this gives $\Ext^h_R(H_j,\omega)=0$ for every $h$ with $h\le i-j-1$. In (\ref{SJD-doo-b}), we have $\HH_{\mathfrak m}^k(C_j)=0$ whenever $$k\le \min\{d,d-r+j+1\}-1=d-\max\{1,r-j\},$$ which gives $\Ext^h_R(C_j,\omega)=0$ for every $h$ with $\max\{1,r-j\}\le h$.\end{proof}

Our first application of Lemma~\ref{SPD} is the next result, Proposition~\ref{localCoh}, which relates local cohomology modules 
along complexes and yields bounds on the top degree of such modules. Proposition~\ref{localCoh} is essentially known;
the idea goes  back to  Gruson, Lazarsfeld, and Peskine \cite[1.6]{GLP}, at least. 
Proposition~\ref{localCoh} has found applications in \cite{KPU-BA}, where we determine the implicit equations defining Rees rings 
of linearly presented grade three Gorenstein ideals.  

\bigskip
Recall the numerical functions of \ref{numerical-functions} and \ref{ainvariant}.

\begin{proposition}\label{localCoh}
Let $R$ be a   
non-negatively graded Noetherian algebra over a local ring $R_0$ with $\dim  R=d$,  $\mathfrak m$ be the maximal homogeneous ideal of $R$,  
$M$ be a 
 graded 
$R$-module, and \[C_{\bullet}: \qquad \ldots \  \longrightarrow C_1 \longrightarrow C_0 \longrightarrow 0 \]be a homogeneous complex of finitely generated  graded
$R$-modules with $\HH_0(C_{\bullet})=M$.
Fix an integer $i$.
Assume that
\begin{enumerate}[\rm(1)]
\item \label{localCoh-b}
$\dim \HH_j(C_{\bullet})\le j+i\ $ for all $j$ with $1\le j\le d-i-1$, and 
\item\label{localCoh-a} 
$j+i+1\le \depth C_j \ $ for all $j$ with $0\le j\le d-i-1$.
\end{enumerate}
Then \begin{enumerate}[\rm(a)]\item\label{localCoh-1}
$\HH^i_{\mathfrak m}(M)$ is a graded subquotient of  $\,  \HH^d_{\mathfrak m}(C_{d-i})$, and
\item \label{localCoh-2} $a_i(M)\le b_0 (C_{d-i})+ a(R)$.  
\end{enumerate}\end{proposition}

\smallskip

\begin{remark}\label{rmk-3} If
\begin{equation}\label{easy-b}\HH_{j}(C_\bullet)_\mathfrak p=0\  \text{for all $j$ and $\mathfrak p$ with } 
1\le j\le d-i-1,\ 
\mathfrak p\in \operatorname{Spec}(R), \text{ and }
 i+2\le\dim R/\mathfrak p,
\end{equation} then hypothesis (\ref{localCoh-b}) is satisfied. Typically, one applies Proposition~\ref{localCoh} when the modules $C_j$ are maximal Cohen-Macaulay modules (for example, free modules over a Cohen-Macaulay ring), because, in this case, hypothesis (\ref{localCoh-a}) about $\depth C_j$ is automatically satisfied. 

\end{remark}

\begin{proof}  (\ref{localCoh-1})
Completing $R_0$ does not change the local cohomology modules in question; hence we may assume that $R_0$ is complete. We use the notation of \ref{GCM}; most notably, $S$ is a non-negatively graded polynomial ring over a local Gorenstein ring $T$, which we may assume to be complete,
and $R$ is obtained from $S$ by factoring out a homogeneous ideal $\mathcal J$ of height $g$. Since $S$ is a polynomial ring over a Cohen-Macaulay ring, the ideal 
$\mathcal J$ contains a homogenous regular sequence $\underline \alpha$ of length $g$. Now $S/(\underline \alpha)$ and $R$ have the same dimension $d$, and we may safely 
replace the latter ring by the former in order to assume that $R$ is Cohen-Macaulay with $R_0$ complete.

When the hypotheses of Proposition~\ref{localCoh} are inserted into Observation ~\ref{SJD-doo}, one obtains, in particular, 
 that the conditions $U_{d-i-1}$ and $V_{d-i}$ hold; so Lemma~\ref{SPD}.\ref{SPD-E-sq-of-H} guarantees that $\Ext^{d-i}(M,\omega)$ is a graded subquotient of 
$\HH^{d-i}(\Hom(C_{\bullet},\omega))$, which is a graded subquotient of
$\Hom(C_{d-i},\omega)$. Graded duality yields $\HH_\mathfrak m^i(M)$ is a graded subquotient of $\HH_\mathfrak m^d(C_{d-i})$. 

\medskip\noindent(\ref{localCoh-2}) Apply (\ref{localCoh-1}) to see that
$$a_i(M)\le a_d(C_{d-i}).$$
Let $F$ be a finitely generated graded free $R$-module which maps surjectively onto 
$C_{d-i}$ so that $b_0(F)=b_0(C_{d-i})$. The long exact sequence of local cohomology gives a surjection \[\xymatrix{\HH^d_{\mathfrak m}(F)\ar@{->>}[r]&\HH^d_{\mathfrak m}(C_{d-i})},\] which shows that 
$$a_d(C_{d-i})\le a_d(F)=b_0(F)+a(R)=b_0(C_{d-i})+a(R).$$ \end{proof}

\begin{corollary}\label{cor-to-localCoh-1}
Adopt the hypotheses of Proposition~{\rm\ref{localCoh}}, with $i=0$.  Then 
$$[\HH^0_\mathfrak m(M)]_\ell=0 \quad \text{for all $\ell$ with}\quad   \maxgendeg(C_d)+a(R)< \ell.$$
\end{corollary}
 
 \smallskip
 
\begin{Reminder}Keep in mind that the hypotheses of Proposition~{\rm\ref{localCoh}} are satisfied if the conditions of Remark~\ref{rmk-3} are in effect.\end{Reminder} 

\begin{corollary}\label{cor-to-localCoh}
Let $R=k[x_1,\dots,x_d]$ be a standard graded polynomial ring over a field, with maximal homogeneous ideal $\mathfrak m$,  
$$C_{\bullet}:\quad \ \ldots \ \longrightarrow C_2 \longrightarrow C_1 \longrightarrow  C_0\longrightarrow 0$$ be a homogeneous complex of finitely generated graded free $R$-modules, and $M=\HH_0(C_{\bullet})$. Assume that 
$\dim \HH_j(C_\bullet)\le j\, $ for all $1\le j\le d-1$  and that 
the subcomplex 
$$C_{d}\longrightarrow \ \ldots \ \longrightarrow C_0\longrightarrow 0$$ of $C_\bullet$ 
is  $q$-linear for some integer $q$. 
Then $\HH^0_\mathfrak m(M)$ is concentrated in degree $q$; that is, $[\HH^0_\mathfrak m(M)]_\ell=0\, $ for all $\ell$ with $\ell\neq q$.
\end{corollary}
\begin{proof} Apply Corollary~\ref{cor-to-localCoh-1}.
We may assume that $C_d\neq 0$ and we 
  see that 
$[\HH^0_\mathfrak m(M)]_\ell=0$ for  $\ell$ with  $ \maxgendeg(C_d)+a(R)< \ell$. But $$ \maxgendeg(C_d)+a(R)=(q+d)-d=q\, ;$$so $[\HH^0_\mathfrak m(M)]_\ell=0$ for $q<\ell$.
On the other hand, $\HH^0_\mathfrak m(M)$ is a graded submodule of $M$ and 
$[M]_\ell=0$ for all $\ell$ with $\ell<q$.
\end{proof}

\bigskip

\section{Bounds on generator degrees of local cohomology modules}\label{TMAS}

In this section, we use Lemma \ref{SPD} and Proposition~\ref{SJD-P} below to prove our main result,
Theorem~\ref{SJD-cor1}. The rest of the section is devoted to applications of this result.
\begin{proposition}\label{SJD-P} Adopt the setup of {\/ \rm \ref{SU-SJD}}. Fix integers $i$ and $t$ with  $1\le t$. Assume\begin{enumerate}[\rm(1)]
\item\label{SJD-P.1} $\HH^{\ell}_{\mathfrak m}(M)=0\ $ for all $\ell$ with $d-i+1\le \ell \le d-i+t-1$,
\item\label{SJD-P.2} $\dim\HH_j(C_{\bullet})\le d-i+j\ $ for all $j$ with $1\le j\le i-1$,
\item\label{SJD-P.3} $\min\{d,d-i+t+j+1\}\le \depth C_j\ $ for all $j$ with $0\le j\le i-1$, and
\item\label{SJD-P.4} $i-j+1\le \depth C_j^\vee\ $ for all $j$ with $i-t+1\le j\le i-1$.
\end{enumerate}
Then there is a natural homogeneous injection
$$\xymatrix{\socle(\Ext^i_R(M,\omega_R))\ar@{^{(}->}[r]&\Ext^{t}_R(k,\operatorname{im}(\partial^\vee_{i-t+1}))}.$$ Moreover,  $t\le \depth(\operatorname{im}(\partial^\vee_{i-t+1}))$, and equality holds if $\depth \Ext^i_R(M,\omega_R)=0$.
\end{proposition}

\begin{remark}\label{rmk-4}Hypothesis (\ref{SJD-P.1}) is always satisfied when $t=1$, or $ \depth M/\HH^0_{\mathfrak m}(M)\ge d-i+t$,  or, in particular, 
$M/\HH^0_{\mathfrak m}(M)=0$. A slightly modified proof shows that condition (\ref{SJD-P.1}) can be replaced by the weaker assumption
\[ \depth {\rm Ext}^{\ell}_R(M, \omega)\ge i+2-\ell \mbox{   \ for all } \ell  \mbox { with } i-t+1\le \ell \le i-1\,.
\]

\noindent
Hypothesis (\ref{SJD-P.2}) is satisfied for $i\le d$  if
$$\HH_{j}(C_\bullet)_\mathfrak p=0\  \text{ for all $j$ and $\mathfrak p$ with } 
1\le j\le i-1 \text{ and }
2\le \dim R/\mathfrak p\, .
$$  As observed in Remark~\ref{rmk-3},  one typically applies Proposition~\ref{SJD-P} when the modules $C_j$ are maximal Cohen-Macaulay modules, because, in this case, hypotheses (\ref{SJD-P.3}) and (\ref{SJD-P.4}) are automatically satisfied for $t\le d$. 
\end{remark}

\begin{proof} Hypotheses (\ref{SJD-P.2}) and (\ref{SJD-P.3}) and Observation~\ref{SJD-doo} imply that 
\begin{align}\label{May15a} &\text{the condition $U_{h}$ of {\ref{SU-SPD}} holds when $h\le i-1$}\\
 %
%
\label{May15b} &\text{the condition
$V_{h}$ of {\ref{SU-SPD}} holds when 
$i-t\le h\le i$.} 
\end{align}
Use (\ref{May15a}), (\ref{May15b}), and  
 Lemma~\ref{SPD}.\ref{SPD-E-in-H} in order to conclude that there is a natural homogeneous injection \begin{equation}\label{May15c}\xymatrix{\Ext^i(M,\omega)\ar@{^{(}->}[r]& \HH^i(C_{\bullet}^\vee)}.\end{equation} 
Combine (\ref{May15c}) and  the natural inclusion 
 $\xymatrix{\HH^i(C^\vee_\bullet) \ar@{^{(}->}[r]& \coker (\partial_i^\vee)}$ 
to see that
 $$\xymatrix{\socle(\Ext^i_R(M,\omega))\ar@{^{(}->}[r]& \socle(\coker (\partial_i^\vee))}.$$

Assumption (\ref{SJD-P.1}) 
implies that $$\Ext^h_R(M,\omega)=0\ \text{ for $i-t+1\le h\le i-1$.}$$ 
Use (\ref{May15a}), (\ref{May15b}),  and Lemma~\ref{SPD}.\ref{SPD-iso} to conclude that  $\, \HH^h(C_{\bullet}^\vee)\simeq \Ext^h(M,\omega)\, $ for $i-t+1\le h\le i-1$;  and therefore $$\HH^h(C_{\bullet}^\vee)=0\ \text{ for $i-t+1\le h\le i-1$.}$$
It follows that the complex 
\begin{equation}\label{the exact complex}C_{i-t}^\vee\xrightarrow{\partial_{i-t+1}^\vee}C_{i-t+1}^\vee \xrightarrow{\partial_{i-t+2}^\vee}\ \cdots \ \xrightarrow{\partial_{i-1}^\vee} 
C_{i-1}^\vee\xrightarrow{\partial_{i}^\vee}C_{i}^\vee \xrightarrow{}
\coker (\partial_i^\vee)\to 0\end{equation} is exact. 
Notice that $1\le \depth C_i^\vee$ since $1 \le \depth \omega_R$. Hence the inequality in Assumption (\ref{SJD-P.4}) holds for $j$ with $i-t+1\le j \le i$. It follows that 
 $\, \Ext^h_R(k,C_{i-h}^\vee)=0\, $ for $0\le h\le t-1$. Long exact sequences associated to $\Ext^{\bullet}_R(k,-)$ then show that
$$\xymatrix{\socle (\coker(\partial_i^\vee))\simeq \Hom_R(k,\coker(\partial_i^\vee))
\ar@{^{(}->}[r]& \Ext^t_R(k,\operatorname{im}(\partial_{i-t+1}^\vee)).}$$

Moreover, $t\le \depth (\operatorname{im} (\partial_{i-t+1}^\vee))$
by the exact complex (\ref{the exact complex}).
 Equality holds unless $$\Ext^t_R(k,\operatorname{im}(\partial_{i-t+1}^\vee))=0,$$ which means $\socle(\Ext^i_R(M,\omega))=0$, hence $0<\depth \Ext^i_R(M,\omega)$.\end{proof}

\begin{theorem}\label{SJD-cor1} Adopt the setup of {\/ \rm\ref{SU-SJD}} and the hypotheses of {\/ \rm\ref{SJD-P}}.
Let $a(R)$ denote the $a$-invariant of $R$ and let $\Theta_t$ be the invariant of $R$ from {\rm\ref{Theta-sub-t}}. 
 If $\, \HH^{d-i}_\mathfrak m(M)$ is finitely generated as an $R$-module, then
$$b_0(\HH^{d-i}_\mathfrak m(M))\le \maxgendeg(C_{i-t})+\Theta_t+a(R).$$
\end{theorem}
\begin{Reminder}Keep in mind that the hypotheses of Proposition~{\rm\ref{SJD-P}} are satisfied if the conditions of Remark~\ref{rmk-4} are in effect.\end{Reminder}

\begin{remark} The assumption in Theorem~\ref{SJD-cor1} that $\, \HH^{d-i}_\mathfrak m(M)$ is finitely generated is satisfied
  if $i=d$ because $\, \HH^{0}_\mathfrak m(M)$ is a graded submodule of $M$. From graded duality it follows that the same assumption is also satisfied if  
  $\dim R_{\p} -\depth M_{\p} <i\, $ for all $\p\in \Supp (M)\setminus \{\m\}$. The last condition holds, for instance, if
   $R_0$ is universally catenary, ${\rm Supp}_R (M)$ is equidimensional of dimension $\not=d-i$, and $M$ is Cohen-Macaulay locally on the punctured spectrum of $R$. 
\end{remark}
\begin{proof}
We may assume that $R_0$ is complete. Applying graded duality and Hom-tensor adjointness twice one obtains isomorphisms of graded $k$-vector spaces
$$\begin{array}{lll}\Hom_R(k, \Ext^i_R(M,\omega))&\simeq& ^*\Hom_{R_0}(k\otimes_R\HH_{\mathfrak m}^{d-i}(M),E_{R_0}(k))\\
&\simeq&^*\Hom_{R}(\HH_{\mathfrak m}^{d-i}(M), \Hom_{R_0}(k,E_{R_0}(k)))\\&\simeq&^*\Hom_{R}(\HH_{\mathfrak m}^{d-i}(M),k).\end{array}$$
It follows that 
\begin{equation}\label{adjointness}^*\Hom_k(k\otimes_R\HH_{\mathfrak m}^{d-i}(M),k)\simeq \socle(\Ext^i_R(M,\omega)).\end{equation}
\noindent
Since $\HH_{\mathfrak m}^{d-i}(M)$ is a finitely generated graded $R$-modules, the graded Nakayama lemma gives
\[\maxgendeg(\HH_{\mathfrak m}^{d-i}(M))=\maxgendeg(k\otimes_R\HH_{\mathfrak m}^{d-i}(M))\,.
\]
Now from (\ref{adjointness}) one attains 
$$\maxgendeg(\HH_{\mathfrak m}^{d-i}(M))=-\indeg(\socle ( \Ext^i_R(M,\omega))).$$
Thus, it remains to prove that
$$  \indeg(\socle(\Ext^i_R(M,\omega))\ge -\maxgendeg(C_{i-t})-\Theta_t-a(R).$$

From Proposition~\ref{SJD-P} we have a homogeneous embedding
$$\xymatrix{\socle(\Ext^i_R(M,\omega)) \ \ar@{^{(}->}[r]&\Ext^t_R(k,N),}$$ where $N=\operatorname{im}(\partial_{i-t+1}^\vee)$. By the same proposition, $\depth N\ge t$. Hence conditions $U_{t-1}$ and $U_{t}$ of {\ref{SU-SPD}} are both satisfied for $C_{\bullet}:=K_{\bullet}$, the Koszul complex of a minimal homogenous generating set of $\m$ as in  (\ref{Theta-sub-t}). The conditions $V_{t-1}$ and $V_t$ are trivially satisfied as $K_{\bullet}$ is a complex of free modules. Therefore Lemma~\ref{SPD}.\ref{SPD-iso} gives the well-known identification
\[\Ext^t_R(k,N)\simeq\HH^t(\Hom_R(K_{\bullet},N))\,.
\]

The latter homology module is a homogenous subquotient of $\Hom_R(K_t,N)$, which in turn is an epimorphic image of $\Hom_R(K_t, C_{i-t}^\vee)$. Let $F$ be a finitely generated graded free $R$-module, with $\maxgendeg (F)=\maxgendeg (C_{i-t})$, that maps homogeneously onto $C_{i-t}$. We obtain a homogeneous inclusion $\xymatrix{C_{i-t}^\vee\ar@{^{(}->}[r]&F^\vee}$. 
Thus:
$$\xymatrix{
&\Hom_R(K_t, C_{i-t}^\vee)
\ar@{^{(}->}[r]\ar@{->>}[d]&\Hom_R(K_t, F^\vee)
\\&\Hom_R(K_t,N).
}$$
We conclude that $\socle(\Ext^i_R(M,\omega))$ is a homogeneous  subquotient of $\Hom_R(K_t, F^\vee)$. Therefore,

$$\begin{array}{lll}\indeg( \socle (\Ext^i_R(M,\omega)))&\ge& \indeg \Hom_R(K_t, F^\vee)\\&=&\indeg F^\vee-\Theta_t\\&=&-\maxgendeg (F)+\indeg \omega-\Theta_t\\&=&-\maxgendeg (C_{i-t})-a(R)-\Theta_t.\end{array}$$
\end{proof}


Corollary~\ref{reg_i} is a self-contained reformulation of Theorem~\ref{SJD-cor1}. The purpose of this reformulation is to obtain one simultaneous bound for $\topdeg \HH_{\mathfrak m}^r(M)$ (which is the subject of Proposition~\ref{localCoh}) and $b_0(\HH_{\mathfrak m}^s(M))$ (which is the subject of Theorem~\ref{SJD-cor1}) for appropriately related $r$ and $s$. We resume this theme in Corollary~\ref{partial-reg-}.
\begin{corollary}\label{reg_i} Adopt the setup of {\/\rm\ref{SU-SJD}}. 
Fix integers $i$ and $t$ with  $1\le t\le \mu(\m)$. Let $a(R)$ denote the $a$-invariant of $R$ and let $\Theta_t$ be the invariant of $R$ from {\rm\ref{Theta-sub-t}}. 
Assume\begin{enumerate}[\rm(1)]
\item\label{SJD-P.1-2}  $\HH^{\ell}_{\mathfrak m}(M)=0\ $ for all $\ell$ with $i-t+1\le \ell \le i-1$,
\item\label{SJD-P.2-2} $\dim\HH_j(C_{\bullet})\le j+i-t$ \ for all $j$ with $1\le j\le d-1-i+t$, and
\item\label{SJD-P.3-2} $C_j$ is a maximal Cohen-Macaulay module for all $j$ with $0\le j\le d-1-i+t$.
\end{enumerate}  If $\, \HH^{i-t}_\mathfrak m(M)$ is finitely generated as an $R$-module, then
$$\sup\{ b_0(\HH^{i-t}_{\mathfrak m}(M)) -\Theta_t,a_i(M)\}\le b_0 (C_{d-i})+ a(R).$$ 
\end{corollary}

\begin{proof} Apply Theorem~\ref{SJD-cor1} (with $i$ replaced by $d-i+t$) and 
Proposition~\ref{localCoh}.\ref{localCoh-2}.
\end{proof}

Various forms of partial regularity appear in the literature; see, for example,
\cite{HT,C12,N}. 
Recall from \ref{aiM} that
$$\reg M=\sup\{a_i(M)+i\mid 0\le i\}.$$
The number $\reg(M/\HH^0_{\mathfrak m}(M))$ which appears on the left side of 
(\ref{partial-reg-.gts}) is equal to the {\it partial regularity} 
\begin{equation}\label{pr}\sup\{a_i(M)+i\mid 1\le i\}\end{equation}
of $M$. We obtain the right side of (\ref{partial-reg-.gts}) as an upper bound for the partial regularity (\ref{pr}) of $M$. Then we show that the 
maximal generator degree of the 
submodule $\HH^0_{\mathfrak m}(M)$ of $M$ that is ignored in the calculation of (\ref{pr}) satisfies the same bound.

We also offer the following interpretation of the right side of (\ref{partial-reg-.gts}). If $C_{\bullet}$ had been a minimal homogeneous resolution 
of $M$ by free $R$-modules, then $b_0(C_i)$ would equal $b_i(M)$. Of course, $$\sup\{b_i(M)-i\mid 0\le i\le d-t\}$$ is another partial regularity of $M$. 
\begin{corollary}\label{partial-reg-} Adopt the setup of {\/ \rm\ref{SU-SJD}} and assume in addition that $R$ is a 
standard graded ring with $R_0$ a field. 
Fix an  integer $t$ with  $1\le t\le \mu(\m)$. Assume\begin{enumerate}[\rm(1)]
\item\label{SJD-P.1-3} $t\le \depth M/\HH^0_{\mathfrak m}(M)$,
\item\label{SJD-P.2-3} $\dim\HH_j(C_{\bullet})\le j\ $ for all $j$ with $1\le j\le d-1$, and
\item\label{SJD-P.3-3} $C_j$ is a maximal Cohen Macaulay module for all $j$ with $0\le j\le d-1$.
\end{enumerate} Then
\begin{equation}\label{partial-reg-.gts}\sup\{ b_0(\HH^0_{\mathfrak m}(M)), \reg (M/\HH_{\mathfrak m}^0(M))\}\le \sup\{ b_0(C_{i})-i\mid 0\le i \le d-t\}+
\reg(R).\end{equation} 
\end{corollary}

\begin{proof}The assumption on  the depth of $M/\HH^0_{\mathfrak m}(M)$ implies that $$\HH^i_{\mathfrak m}\big(M/\HH^0_{\mathfrak m}(M)\big)=0\quad\text{for $0\le i \le t-1$}.$$ 
Therefore
\begin{align} \reg (M/\HH_{\mathfrak m}^0(M))&=  \sup\{ a_i(M)+i \mid t\le i \le d\}\notag\\
&\le \sup\{b_0(C_{d-i})+a(R)+i\mid t\le i\le d\}&&\text{by Proposition~\ref{localCoh}.\ref{localCoh-2}}\notag\\
&= \sup\{b_0(C_{i})+a(R)+d-i\mid 0\le i\le d-t\}.\notag\intertext{
Apply Theorem~\ref{SJD-cor1} with $i=d$ to obtain}
b_0(\HH^0_{\mathfrak m}(M))&\le b_0(C_{d-t})+t+a(R).\notag\end{align}The invariant $\Theta_t$ from \ref{SJD-cor1} and \ref{Theta-sub-t} is $t$ because $R$ is standard graded, and $a(R)+d=\reg(R)$ because $R$ is standard graded and Cohen-Macaulay.\end{proof}

\smallskip

Corollary~\ref{SJD-cor3} is well known and easy to prove if $R$ is a standard graded polynomial ring over a field because, in this case, the sum $\Theta_d + a(R)$ is zero, the  maximal Cohen-Macaulay 
module $M/\HH_{\mathfrak m}^0(M)$ is free, $\HH_{\mathfrak m}^0(M)$ is a direct summand of $M$, and it is true and clear that
$$\maxgendeg (\HH_{\mathfrak m}^0(M))\le \maxgendeg(M).$$ On the other hand, the result in the stated generality is new and intriguing.

\begin{corollary}\label{SJD-cor3} Let $R$ be a non-negatively graded Cohen-Macaulay  ring with $R_0$ a local ring. Denote the maximal homogeneous ideal of $R$ by $\mathfrak{m}$ and $\dim R$ by $d$. Let $M$ be a finitely generated graded $R$-module and assume that $M/\HH_{\mathfrak m}^0(M)$ is a maximal Cohen-Macaulay $R$-module. Then
$$\maxgendeg (\HH_{\mathfrak m}^0(M))\le \maxgendeg (M)+\Theta_d + a(R),$$where $\Theta_d$ is defined in {\rm\ref{Theta-sub-t}} and $a(R)$ is the $a$-invariant of $R$. 
\end{corollary}

\begin{proof} Again we may assume that $1\le d$. Apply Theorem~\ref{SJD-cor1} with $i=t=d$ and $C_{\bullet}$ a free resolution of $M$.
\end{proof}

\smallskip
Corollary~\ref{SJD-cor2} is the numerical consequence of Theorem~\ref{SJD-cor1} that we apply most often.
\begin{corollary}\label{SJD-cor2} Adopt the setup of {\/\rm\ref{SU-SJD}} and assume in addition that $R=k[x_1,\dots,x_d]$ be a standard graded polynomial ring over a field.
Assume that $\dim \HH_j(C_\bullet)\le j$ whenever $1\le j\le d-1$ and that $\min\{d,j+2\}\le \depth C_j$ whenever $0\le j\le d-1$. Then  $$b_0(\HH_{\mathfrak m}^0(M))\le \maxgendeg (C_{d-1})-d+1.$$
\end{corollary}
\begin{proof} Apply Theorem~\ref{SJD-cor1} with $i=d$ and $t=1$. \end{proof}

The next result is analogous to Corollary~\ref{cor-to-localCoh}. The hypothesis is weaker than the hypothesis in Corollary~\ref{cor-to-localCoh} because we do not require that the complex $C_{\bullet}$ be linear quite as far in the present result. Alas, the conclusion is also weaker. We conclude that the generators of $\HH^0_\mathfrak m(M)$ are concentrated in one degree rather than learning that all of  $\HH^0_\mathfrak m(M)$ is concentrated in one degree.

\begin{corollary}\label{SJD-Apr8} Let $R=k[x_1,\dots,x_d]$ be a standard graded polynomial ring over a field, with maximal homogeneous ideal $\mathfrak m$,  
$$C_{\bullet}:\quad \ \ldots \ \longrightarrow C_2 \longrightarrow C_1 \longrightarrow  C_0\longrightarrow 0$$ be a homogeneous complex of finitely generated graded free $R$-modules, and $M=\HH_0(C_{\bullet})$. Assume that 
$\dim \HH_j(C_\bullet)\le j\, $ for all $\/1\le j\le d-1$  and that 
the subcomplex 
$$C_{d-1}\longrightarrow \ \ldots \ \longrightarrow C_0\longrightarrow 0$$ of $C_\bullet$ is $q$-linear for some integer $q$. Then every minimal homogeneous generator of $\HH^0_\mathfrak m(M)$ has degree $q$.
\end{corollary}
\begin{proof} We may assume that $1\le d$ since otherwise $\HH^0_\mathfrak m(M)=M$. Apply Corollary~\ref{SJD-cor2}  to conclude
$$b_0(\HH_{\mathfrak m}^0(M))\le \maxgendeg (C_{d-1})-d+1=q.$$
On the other hand, $\HH_{\mathfrak m}^0(M)$ is a submodule of $M$  and every minimal homogeneous generator of $M$ has degree $q$. \end{proof}

\smallskip

\section{Geometric Applications.} 
\label{dim-one}

We apply the local cohomology techniques of Section~\ref{TMAS} to draw conclusions about the generator degrees of the second  symbolic power of the prime ideal which defines a monomial curve in affine space;  of the 
second  symbolic power of the  ideal which defines a finite set of points  in projective space; and of the saturated ideal defining the intersection of  a projective scheme with a general linear subspace.

Recall that if $I$ is an ideal  in a Noetherian ring $R$, then the $t$-th symbolic power of $I$ is $I^tR_W\cap R$, where $W$ is the
complement of the union of the associated primes of $I$ and $R_W$ is the localization of $R$ at the multiplicative
system $W$; see \cite{HH}. 
 The first two applications in this section make use of local cohomology by way of the following lemma.

\begin{lemma}\label{doo} Let $P$ be a non-negatively graded Noetherian   ring with $P_0$ an Artinian local ring. 
Let  $I$ be a homogeneous ideal in $P$ 
with $P/I$ a Cohen-Macaulay ring of dimension one. Write $R=P/I$ and denote the maximal homogeneous ideal of $R$ by $\mathfrak m$.
Then $$\maxgendeg (I^{(2)})\le \sup\{\maxgendeg(I)+b_0(\mathfrak m)+a(R), 2\mgd(I)\}.$$ 
\end{lemma}
\begin{proof}  Let $M$ be the $R$-module $I/I^2$. Notice that $\HH^0_\mathfrak m(M)=I^{(2)}/I^2$. If $M/\HH^0_\mathfrak m(M)$ is the zero module, then $I^{(2)}=I$
and the degree bounds hold automatically. Otherwise,   $M/\HH^0_\mathfrak m(M)$ has positive depth and   is a maximal Cohen-Macaulay $R$-module. 
Apply Corollary~\ref{SJD-cor3} to the $R$-module $M$ to conclude
that \begin{equation}\label{n-m-y}\maxgendeg (\HH^0_\mathfrak m(I/ I^2))\le \maxgendeg (I/I^2)+\Theta_1+a(R).\end{equation} Nakayama's Lemma guarantees that $\maxgendeg (I/I^2)=\maxgendeg (I)$.
In the present situation, $\Theta_1$, which is defined in (\ref{Theta-sub-t}),
is equal to $b_0(\mathfrak m)$. Thus,
\begin{align}\notag\maxgendeg (I^{(2)})\le \sup \left\{
\maxgendeg \left({I^{(2)}}/{I^2}\right),\maxgendeg(I^2) \right\}
&=\sup \left\{\maxgendeg \left(\HH^0_\mathfrak m(I/I^2)\right),\maxgendeg(I^2) \right\}\\&\le \sup \{
\maxgendeg (I)+b_0(\mathfrak m) +a(R),2\maxgendeg(I) \}.
\notag\end{align}\end{proof}

Our first application of Lemma~\ref{doo} is to monomial curves. The hypothesis in Corollary~\ref{monomialCurve} that $H$ is a numerical semigroup includes the requirement that all large positive integers are in $H$. The Frobenius number of $H$, denoted $F(H)$, is the largest integer $b$ with $b\notin H$. One way to see the connection between $F(H)$ and the language of Lemma~\ref{doo} is described below.

Let  $R$ be  a
non-negatively graded ring over a field. Denote    the maximal homogeneous  ideal of $R$ by $\mathfrak m$ and the dimension of $R$ by $d$.
A well known theorem of Serre, see for example \cite[4.3.5]{BH}, shows   that
\begin{align}
 &\max\{n \, | \,  \text{Hilbert Function}_R(n)\neq \text{Hilbert Quasi-Polynomial}_R(n)\}\notag \\= \ \, &\max\left\{n \, \left| \ \sum\limits_{i=0}^d(-1)^i\dim \HH^i_\mathfrak m(R)_n\neq 0\right.\right\}.\label{PostNum}\end{align}
If   $R$ is Cohen-Macaulay, then the
number on the right side of (\ref{PostNum}) is equal to the $a$-invariant of $R$. If $R\subset k[t]$ is the coordinate ring of a monomial curve over an infinite field, then  the number on the left side of (\ref{PostNum}) represents the largest exponent $n$ with $t^n\notin R$. If $R$ is a standard graded ring, then the number on the left is often called the postulation number of $R$.

\begin{corollary}\label{monomialCurve} 
Let $k$ be an infinite field, $H$ be a numerical semigroup minimally generated by the positive integers $h_1<h_2<\cdots <h_\ell$,
 $P$ be the polynomial ring $k[x_1,\dots,x_\ell]$, and $\mathfrak p\subset P$ be the prime ideal which defines the monomial curve $$\{(\tau^{h_1},\dots,\tau^{h_\ell})\subset \mathbb A_k^\ell\mid \tau\in k\}.$$ Then the maximal generator degree of the second symbolic power of $\, \mathfrak p$ satisfies
$$\maxgendeg (\mathfrak p^{(2)})\le \sup\{\maxgendeg (\mathfrak p)+\text{ the maximal generator of $H$} +\text{ the Frobenius number of $H$}, 2\mgd (\mathfrak p)\}.$$
\end{corollary}
\begin{proof} 
 View $P$ as a graded ring with $\deg x_i=h_i$.  Then $\mathfrak p$ is the kernel of the homogeneous ring homomorphism $P\to k[t]$ with $x_i\mapsto t^{h_i}$. 
Let $R=P/\mathfrak p$ and $\mathfrak m$ denote the maximal homogeneous ideal of $R$.
 We see that  $R$ is a one-dimensional Cohen-Macaulay domain,  the $a$-invariant of $R$ is the Frobenius number of $H$, and $\mgd(\mathfrak m)=h_\ell$ is the maximal generator of $H$. The assertion follows from Lemma~\ref{doo}.
\end{proof}

\begin{example}In the language of Corollary~\ref{monomialCurve}, if $H={<}3,4,5{>}$, then $$\maxgendeg (\mathfrak p)+b_0(\mathfrak m)+F(H)=10+5+2=17,\quad 2\maxgendeg (\mathfrak p)=20,
\quad \text{and} \quad \maxgendeg(\mathfrak p^{(2)})=18;$$so there are situations where some minimal generator of $\mathfrak p^2$ of degree more than $$\maxgendeg (\mathfrak p)+b_0(\mathfrak m)+F(H)$$ is also a minimal generator of $\mathfrak p^{(2)}$. 
The calculation of $\maxgendeg(\mathfrak p^{(2)})$ was made in Macaulay2 \cite{M2} over the field of rational numbers.  \end{example}

\smallskip
 Our second application of Lemma~\ref{doo} concerns the ideal of a finite set of points in projective space and to other similar ideals. In the situation of Corollary~\ref{typed-by-Claudia},
$$\reg(R)=a(R)+1=\text{ the postulation number of $P/I$ plus one};$$
see the discussion surrounding (\ref{PostNum}).

\begin{corollary}\label{typed-by-Claudia} Let $P$ be a standard graded polynomial ring over a field 
and  $I$ be a homogeneous ideal in $P$ 
with $R=P/I$ Cohen-Macaulay of dimension one. Then 
\[b_0(I^{(2)})\le b_0(I)+\reg(R) +1.\]
Furthermore, if, in addition to the above hypotheses, the minimal homogeneous resolution of $I$ by free $P$-modules 
is not  linear, then 
\begin{equation}\label{8.4.1}b_0(I^{(2)})\le b_0(I)+\reg(R).\end{equation}
\end{corollary}

\begin{proof} The ring $R$ is a standard graded  ring over a field; so every minimal generator of the maximal homogeneous ideal $\mathfrak m$ of $R$ has degree one; furthermore $1+a(R)=\reg(R)$. Apply Lemma~\ref{doo} to obtain
\begin{align}\maxgendeg (I^{(2)})\le \sup\{\maxgendeg(I)+b_0(\mathfrak m)+a(R), 2\mgd(I)\}&\notag \le \sup\{\maxgendeg(I)+\reg(R), 2\mgd(I)\}
\\&=
\maxgendeg(I)+\max\{\reg(R),\mgd(I)\}
.\notag\end{align}
In the general case, 
$$\maxgendeg (I)\le \operatorname{reg}(I) \leq \operatorname{reg}(R)+1.$$On the other hand, if 
the minimal homogeneous resolution of $I$ by free $P$-modules 
is not  linear, then $$b_0(I)<\reg(I) \leq \operatorname{reg}(R)+1;$$ hence $b_0(I)\le \reg(R)$.
\end{proof}
\begin{example}  If $P=\mathbb Q[x,y,z]$ and $I$ is the ideal of $P$ generated by the $2\times 2$ minors of 
$$\bmatrix x&y&z\\y&z&x\endbmatrix,$$ then $R=P/I$ is a one dimensional Cohen-Macaulay ring, the minimal homogeneous resolution of $I$ by free $P$-modules is linear,
$$b_0(I)+\reg(R)=2+1=3, \quad\text{and}\quad b_0(I^{(2)})=4;$$ so the inequality (\ref{8.4.1}) does not hold in the general case.
Again, $b_0(I^{(2)})$ was computed using Macaulay2 \cite{M2}. 
\end{example}

\begin{remark}If one re-does the calculation of Corollary~\ref{typed-by-Claudia} starting at (\ref{n-m-y}), then one can  
read the conclusion of Corollary~\ref{typed-by-Claudia}
as \begin{equation}\label{formulation}b_0(I^{(2)}/I^2)\le b_0(I)+\reg(R).\end{equation}
Indeed, 
$$b_0(I^{(2)}/I^2)=\maxgendeg (\HH^0_\mathfrak m(I/ I^2))\le \maxgendeg (I/I^2)+\Theta_1+a(R)
=b_0(I)+1+a(R)
=b_0(I)+\reg(R).$$
The formulation (\ref{formulation}) affords a direct  comparison with the relevant part of \cite[Cor~7.8]{EHU}:
$$\topdeg(I^{(2)}/I^2)\le b_1(I)-1+\reg(R).$$ 
Observe that $b_0(I^{(2)}/I^2)\le \topdeg(I^{(2)}/I^2)$ and $b_0(I)\le b_1(I)-1$. (Recall the meaning of $b_i$ from (\ref{numerical-functions}).)
\end{remark}

\medskip

Corollary~\ref{hyperplane} is about hyperplane sections of subschemes of projective space. For instance, let $V$ be the subscheme of $\mathbb P^{d-1}_k$ defined by the homogeneous ideal $I$ in $R=k[x_1,\dots,x_d]$ and $H$ be a linear subspace of $\mathbb P^{d-1}_k$ defined by  general linear forms in $k[x_1,\dots,x_d]$. We produce an upper bound for the maximal generator degree of the saturated ideal defining  $V\cap H$, in terms of information that can be read from a single shift in the minimal homogeneous resolution of $R/I$. 
The analogous bound for the highest degree of a form that is in the saturated ideal of $V\cap H$ but not in the image of $I$ was proved in \cite[5.1]{EHU}.
Notice that the saturated ideal of 
$V\cap H$ is the ideal of polynomials vanishing on $V\cap H$ if $k$ is algebraically closed and $I$ is radical \cite[5.2]{Fl}.
\begin{corollary}\label{hyperplane}Let $R=k[x_1,\dots,x_d]$ be a standard graded polynomial ring over a field $k$, $I$ be a homogeneous ideal of $R$, 
and $L$ be an ideal minimally generated by $c$ linear forms in $R$. Assume that $\dim \, {\rm Tor}^R_1(R/I, R/L)\le 1$.
Let $\bar I$ be the image of $I$ in $\bar R=R/L$ and $J$ be the saturation $J=\bar I^{\rm sat}$ of $\bar I$. Then $$b_0(J)\le \max\{b_0(I), b_{d-c-2}(I)-d+c+1\}.$$ 
Furthermore, if $c=\dim (R/I)-1$, then $$b_0(J)\le  b_{d-c-2}(I)-d+c+2.$$ \end{corollary}

  \smallskip
  
\begin{remark} The inequality $\dim \, {\rm Tor}^R_1(R/I, R/L)\le 1$ is satisfied if $\dim (R/I)\le 1$. It also holds if $k$ is infinite and $L$ is generated by general linear forms, because such forms are a filtered regular sequence on $R/I$, see \cite[2.3]{Xie} for a proof. 
\end{remark}

\begin{proof}If $d-1\le c$, then $\bar R$ is a principal ideal domain and hence $J=\bar I$. Thus we may assume that $c\le d-2$. Denote the maximal homogeneous ideal of $R$ by $\mathfrak m$ and the maximal homogeneous ideal of $\bar R$ by $\bar {\mathfrak m}$. The ideal $J$ of $\bar R$ 
is equal to $$J=\bigcup_i(\bar I:_{\bar R}\bar{\mathfrak m}^i);$$
 therefore $$J/\bar I= \HH^0_{\bar{\mathfrak m}} (\bar R/\bar I)$$ and
$$b_0(J)\le \sup\{b_0(\bar I),b_0(\HH^0_{\bar{\mathfrak m}} (\bar R/\bar I))\}
\le \sup\{b_0(I),b_0(\HH^0_{\bar{\mathfrak m}} (\bar R/\bar I))\}.$$

We now bound $b_0(\HH^0_{\bar{\mathfrak m}} (\bar R/\bar I)$.
 Let $C_\bullet$ be a minimal homogeneous resolution  of $R/I$ by free $R$-modules. Consider the complex $\bar C_\bullet=C_\bullet\otimes_R \bar R$. 
 Our assumption on $\operatorname{Tor}_1$ and the rigidity of $\operatorname{Tor}$ \cite[2.1]{A} imply that $\dim \, {\rm Tor}^R_i(R/I, R/L)\le 1$ for all positive $i$.
 Keep in mind that $\bar R$ is a polynomial ring of dimension $d-c$. Apply Corollary~\ref{SJD-cor2} to the complex $\bar C_\bullet$ to obtain 
\begin{align}b_0(\HH^0_{\bar{\mathfrak m}} (\bar R/\bar I))&\le b_0(C_{d-c-1})-(d-c)+1
=b_{d-c-1}(R/I)-d+c+1\notag\\&=b_{d-c-2}(I)-d+c+1.\notag
\end{align}
This completes the proof of the general case.

If $c=\dim (R/I)-1$, then $d-c-1=\htt\, I=\grade I$ and $$b_0(R/I)<b_1(R/I)<\dots < b_{d-c-1}(R/I)$$ because 
$$0\lto C_0^*\lto \/ \ldots \/ \lto C_{d-c-1}^*$$ is a minimal resolution. It follows that
$$b_0(I)<\dots<b_{d-c-2}(I);$$ and therefore, $b_0(I)\le b_{d-c-2}(I)-d+c+2$.
\end{proof}


\bigskip






\begin{thebibliography}{99}

\bibitem{A}M.~Auslander, {\em Modules over unramified regular local rings}, Illinois J. Math {\bf 5} (1961), 631--647.

\bibitem{BH}W.~Bruns and J.~Herzog,  {\em Cohen-Macaulay rings}, Cambridge Studies in Advanced Mathematics {\bf 39}, Cambridge University Press, Cambridge, 1993.





\bibitem{C12}M.~Chardin, {\em Powers of ideals and the cohomology of stalks
and fibers of morphisms}, Algebra Number Theory {\bf 7} (2013),  1--18. 





\bibitem{EHU}D.~Eisenbud, C.~Huneke, and B.~Ulrich, {\em The regularity of Tor and graded Betti numbers}, Amer. J. Math. {\bf128} (2006),  573--605.

\bibitem{Fl} H.~Flenner, {\em Die S\"atze von Bertini f\"ur lokale Ringe}, Math. Ann. {\bf 229} (1977),  97--111. 
 

\bibitem{M2}
D.~R.~Grayson and M.~E.~Stillman, \textit{Macaulay2, a software system for research  in algebraic geometry}, available at {http://www.math.uiuc.edu/Macaulay2/}. 


\bibitem{GLP}L.~Gruson, R.~Lazarsfeld, and C.~Peskine, {\em On a theorem of Castelnuovo, and the equations defining space curves}, Invent. Math. {\bf 72} (1983),  491--506.






\bibitem{HT}
L.~T.~Hoa and T.~N.~Trung, {\em Partial Castelnuovo-Mumford regularities of sums
and intersections of powers of monomial ideals}, Math. Proc.  Camb.
Phil. Soc. {\bf 149} (2010), 229--246.

\bibitem{HH} M.~Hochster and C.~Huneke, {\em
Comparison of symbolic and ordinary powers of ideals},
Invent. Math. 147 (2002),  349--369. 




\bibitem{KPU-BA}A. Kustin, C. Polini, and B. Ulrich, {\em The equations defining blowup algebras of 
 height three Gorenstein ideals}, Algebra Number Theory {\bf 11} (2017),  1489--1525. 





\bibitem{N}W.~Niu {\em Some results on asymptotic regularity of ideal sheaves},  J. Algebra {\bf 377} (2013), 157--172.










\bibitem{Xie} Y.~Xie, {\em Formulas for the multiplicity of graded algebras}, Trans. Amer. Math. Soc. {\bf 364} (2012),  4085--4106. 

\end{thebibliography}
\end{document}